\documentclass[10pt]{article}

\usepackage{amsmath,amsthm}
\usepackage{amssymb}
\usepackage{url}
\usepackage{srcltx}
\usepackage{enumerate}
\usepackage[latin1]{inputenc}
\usepackage{mathrsfs}
\usepackage{color}
\usepackage{curves}
\usepackage{eepic}
\usepackage{epic}
\usepackage{graphics,color}
\usepackage{pstricks}
\usepackage{pstricks-add,pst-node}
\usepackage{subfigure}
\usepackage{pstricks}
\usepackage{pstricks-add,pst-node}
\usepackage{curves}

\renewenvironment{proof}{\par \noindent \textbf{Proof.}
}{\hfill$\Box$\medskip}

\newtheorem{theorem}{Theorem}
\newtheorem{corollary}[theorem]{Corollary}

\newtheorem*{claim}{Claim}
\newtheorem{proposition}[theorem]{Proposition}
\newtheorem{pr}[theorem]{Property}
\newtheorem{lemma}[theorem]{Lemma}

\date{\today}

 \title{Distinguishing Number for Some Circulant Graphs }

  \author{Sylvain GRAVIER\thanks{Institut Fourier - SFR Maths \`{a} Modeler.UMR 5582 CNRS/Universit\'{e} Joseph Fourier \ 
    100 rue des maths, BP 74, 38402 St Martin d'H\`{e}res, France}, Kahina Meslem  \thanks{Laboratoire LaROMaD,
 SFR Maths \`{a} Modeler. Facult\'{e} des Math\'{e}matiques, U.S.T.H.B. 
   El Alia Bab-Ezzouar  16111, Algiers, Algeria}, Souad~ SLIMANI  \footnotemark[2]}
  
\begin{document}

\maketitle


\begin{abstract}
Introduced by Albertson et al. \cite{albertson},
the distinguishing number $D(G)$ of a graph $G$ is the least integer $r$ such that
there is a $r$-labeling of the vertices of $G$ that is not preserved by any nontrivial
automorphism of $G$. Most of graphs studied in literature have 2 as a distinguishing number value except
complete, multipartite graphs or  cartesian product of complete graphs depending on $n$. In this paper, we study  circulant
graphs of order $n$ where the adjacency is defined using a symmetric subset $A$ of $\mathbb{Z}_n$, called generator.
We give a construction of a family of circulant graphs of order $n$ and we show that this class has distinct distinguishing numbers and these lasters are not depending on $n$.
    `
\end{abstract}

%


\section{Introduction}\label{sec:in}
In 1979, F.Rudin \cite{rudin} proposed a problem in Journal of Recreational Mathematics
by introducing the concept of the breaking symmetry in graphs. Albertson et al.\cite{albertson}
studied the distinguishing number in graphs defined as the minimum number of labels
needed to assign to the vertex set of the graph in order to distinguish any non
trivial automorphism graph. The distinguishing number is widely focused in the recent years
: many articles deal with this invariant in particular classes of graphs: trees \cite{tree},
hypercubes \cite{Bogstad}, product graphs \cite{klav_power} \cite{Imrich_cartes_power}
\cite{klav_cliques} \cite{Fisher_1} and  interesting algebraic properties of
distinguishing number were given in \cite{Potanka}  \cite{tym} and \cite{Z}.
Most of non rigid structures of graphs (i.e structures of graphs having at most
one non trivial automorphism)  need just two labels to destroy any non
trivial automorphism. In fact, paths $P_n$ $(n>1)$, cycles $C_n$ $(n>5)$, hypercubes $Q_n$  $(n>3)$, $r$ $(r>3)$
times cartesian product of a graph $G^r$ where $G$ is of order $n>3$, circulant graphs
of order $n$ generated by $\{\pm 1,\pm 2,\dots \pm k\}$ \cite{gravier}($n\geq2k+3$) have 2 as a common
value of distinguishing number. However, complete graphs, complete multipartite graphs
\cite{chrom} and cartesian product of complete graphs (see \cite{klav_cliques} \cite{Fisher_1} \cite{Fisher_2})
are the few classes with a big distinguishing number. The associated invariant increases with the order of the graphs.
In order to surround the structure of a graph of a given order $n$ and get a proper distinguishing number we built
regular graphs $C(m,p)$ of order $mp$ where the adjacency is described  by introducing a generator $A$ $(A \subset \mathbb{Z}_{m.p})$.
These graphs are generated by $A=\{(p-1)+ r.p, (p+1)+ r.p$ : $0\leq r \leq m-1\}$ for all $n=m.p \geq 3$. In fact, the motivation of this paper is to give an answer to this following question, noted ${\mathcal{(Q)}}$:\\
``Given a sequence of ordered and distinct integer numbers $d_1,d_2,\dots,d_r$
in $\mathbb{N}^* \setminus \{1\}$, does it exist an integer $n$ and $r$
graphs $G_i$ $(1\leq i\leq r)$ such that $D(G_i)=d_i$ for all $i=1,\dots,r$  and $n$ is the common order of the $r$ graphs?"\\
In the following proposition, we give the answer to this question:

\begin{proposition}\label{disconnected}
Given an ordered sequence of $r$ distinct integers $d_1,d_2,\dots,d_r$  with $r\geq2$ and $d_i\geq 2$ for $i=1,\dots,r$,
there exists $r$ graphs $G_1,G_2,\dots,G_r$  of order $n$ such that $G_i$ contains a clique $K_{d_i}$ and $D(G_i)=d_i$
for all $1\leq i\leq r$.
\end{proposition}

\begin{proof}
 Suppose that $d_1\neq 2$  and $n=d_r$. For the integer $d_r$, we assume that $G_r \simeq K_{d_r}$ and  $D(G_r)=d_r$.\\
For the other integers, we consider the disconnected $(r-1)$ graphs $G_i$ having two connected component
$C$ and $C'$  such that $C\simeq K_{d_i}$ and
$C'$  is a path $P_{n-d_i}$ for all $i=1,\dots,(r-1)$.\\
Observe that, when  $d_1\neq 2$ or $n= d_r\neq4$, then the connected component $C$ and $C'$ can not be isomorphic. By consequence, an automorphism $\delta$ of a graph
$G_i$ acts in the same connected component for all $1\leq i\leq r-1$. More than, $D(G_i)=\max (D(C),D(C'))=D(C)=d_i$ for all $1\leq i\leq r-1$. \\
If $d_1= 2$ and $n= d_r=4$ the same graphs are considered except for $G_1$ where we put $G_1\simeq P_4$. 
Then, $D(G_1)=2=d_1$.
\end{proof}%

\noindent The graphs of Proposition \ref{disconnected} are not completely satisfying
since these ones are not connected. Furthermore, these graphs
give no additional information for graphs having hight distinguishing number,
since they just use cliques for construction. So our purpose
is to construct connected graphs structural properties that give answer to question ${\mathcal{(Q)}}$
\begin{theorem}\label{main}
Given an ordered sequence of $r$ distinct integers $d_1,d_2,\dots,d_r$
with $r\geq2$ and $d_i\geq 2$ for $i=1,\dots,r$, there exists $r$
connected circulant graphs $G_1,G_2,\dots,G_r$  of order $n$ such that $D(G_i)=d_i$.
\end{theorem}

\noindent So, in section 1, basic definitions and preliminary results used in this paper are given. Then in section 2,
we define circulant graphs $C(m,p)$ , $n=m.p\geq 3$ and provide interesting structural properties of
this class of graphs.These later are used to determine the associated distinguishing number which is given in section 3.
We also give the proof of Theorem \ref{main} in the same section.
Finally, in section 4, we conclude by some remarks and possible improvement of reply of the question ${\mathcal{(Q)}}$.
\section{Definitions and Preliminaries Results}\label{sec:1}

We only consider finite, simple, loopless, and undirected graphs $G=(V ,E)$ where $V$
is the vertex set and $E$ is the edge set.
The \emph{complement} of $G$ is the simple graph $\overline{G}=(V,\overline{E})$ which consists of the same vertex
set $V$ of G. Two vertices $u$ and $v$ are adjacent in $\overline{G}$ if and only
if they are  not in $G$.
The \emph{neighborhood} of a vertex $u$, denoted by $N(u)$, consists in all the vertices $v$ which are adjacent to
$u$.
A \emph{complete graph} of order $n$, denoted $K_n$, is a graph having $n$ vertices such that all
two distinct vertices are adjacent. A \emph{path} on $n$ vertices,
denoted $P_n$, is a sequence of distinct
vertices  and and $n-1$ edges $v_iv_{i+1}$, $1 \leq i \leq n - 1$.
A path relying two distinct vertices $u$ and $v$ in $G$ is said $uv$-path.
A \emph{cycle}, on $n$ vertices denoted $C_n$, is a path with $n$ distinct vertices
$v_1, v_2, \dots, v_n$ where $v_1$ and $v_n$ are confused. For a graph $G$, the \emph{distance} $d_G(u, v)$
between vertices $u$ and $v$ is defined as the number of edges on a shortest
$uv$-path.\\
Given a subset $A \subset \mathbb{Z}_n$ with $0 \not \in A$ and for all $a\in A$ and $-a\in A$, a \emph{circulant graph}, is a graph on $n$ vertices $0,1,\dots,n-1$ where two vertices
$i$ and $j$ are adjacent if $j-i$ modulo $n$ is in $A$.

\noindent The \emph{automorphism} (or \emph{symmetry}) of a graph $G=(V,E)$
is a permutation $\sigma$ of the vertices of $G$ preserving adjacency i.e if $xy \in E$,
then $\sigma(x)\sigma(y) \in E$. The set of all automorphisms of $G$, noted $Aut(G)$
defines a structure of a group. A labeling of vertices  of a graph $G$, $c: V(G) \rightarrow \{1,2,\dots, r\}$ is
said $r$-\emph{distinguishing} of $G$ if $\forall \sigma \in Aut (G)\setminus \{Id_G\}$:
$c \neq c \circ \sigma$. That means that for each automorphism $\sigma \neq id $
there exists a vertex $v\in V$ such that $c(v)\neq c(\sigma(v))$.
A \emph{distinguishing number} of a graph $G$, denoted by $D(G)$, is
a smallest integer $r$ such that $G$  has an $r$-distinguishing labeling.
Since $Aut(G)=Aut(\overline{G})$, we have $D(G)=D(\overline{G})$.
The distinguishing number of a complete graph of order $n$ is equal to $n$.
The distinguishing number of complete multipartite graphs is given in the following theorem:
\begin{theorem} \cite{chrom}\label{multipartite}
Let $K_{a_1^{j_1} ,a_2^{j_2},\dots,a_r^{j_r}}$ denote the complete multipartite graph that has $j_i$
partite sets of size $a_i$ for $i = 1, 2,\dots,r$ and $a_1 > a_2 > \dots > a_r$.
Then $D(K_{a_1^{j_1} ,a_2^{j_2},\dots,a_r^{j_r}})= \min \{p :\binom{p}{a_i} \geqslant j_i$ for all $i \}$
\end{theorem}
Let us introduce the concept of modules useful to investigate distinguishing number in graphs.
A \emph{module} in the graph $G$ is a subset $M$ of vertices which share the same neighborhood outside $M$ i.e
for all $y \in V \setminus M$: $M \subseteq N(y)$ or $xy \not \in E$  for all $x\in M$.
A trivial module in a graph $G$ is either the set $V$ or
any singleton vertex. A module $M$ of $G$
is said \emph{maximal} in $G$ if for each non trivial module $M'$  in $G$ containing $M$, $M'$ is reduced to $M$. The following
lemma shows how modules can help us to estimate the value of distinguishing number in graphs:
\begin{lemma}\label{module}
Let $G$ be a graph and $M$ a module of $G$. Then, $D(G)\geq D(M)$
\end{lemma}
\begin{proof}
 \noindent Let $c$ be an $r$-labeling such that $r<D(M)$.
Since $r<D(M)$, there exits $\delta\mid_{M}$ a non trivial automorphism of $M$
such that $c(x)=c(\delta\mid_{M}(x))$  for all $x \in M$ i.e the restriction
of $c$ in $M$ is not a distinguishing. Now, let $\delta$ be the extension of $\delta \mid_{M}$ to $G$
with $\delta(x)=x$ $\forall x \not \in M$ and $\delta(x)=\delta\mid_M(x)$ otherwise.
We get $c(x)=c(\delta(x))$ for all $x \in G$. Moreover, $\delta \neq id$ since $\delta\mid_{M} \neq id\mid_{M}$.
\end{proof}

\section{Circulant Graphs $C(m,p)$}\label{sec:2}
\noindent In this section, we study distinguishing number of circulant graphs $C(m,p)$ of order $n=m.p\geq3$ with
$m\geqslant 1$ and $p\geqslant 2$. A vertex $i$ is adjacent to $j$ in $C(m,p)$ iff $j-i$ modulo $n$ belongs to
$A=\{p-1+r.p, p+1+ r.p$, $0\leq r \leq m-1\}$ (See Fig. \ref{weakly}).
When $p>1$, these graphs are circulant since for all $0 \leq r\leq  m-1$
the symmetric of $p-1+r.p$ is $1+p+(m-r-2)p$ which belongs to $A$ and $p>1$ implies that $0\notin A$.
By construction, set $C(m,1)$ is the clique $K_m$.
Let specify some other particular values of $p$ and $m$, $C(1,p)$ is the cycle $C_p$.
Also we have: $C(m,2)=K_{m,m}$ and $C(m,3)=K_{m,m,m}$.
By Theorem \ref{multipartite}, $D(C(m,2))=D(C(m,3))=m+1$. Moreover, $D(C(1,p))=2$ for $p\geq6$.
\begin{pr}\label{proper}
The vertex set  of $C(m,p)$ ($m\geqslant 2$ and $p\geqslant 2$) can be partitioned into $p$ stable 
modules $M_i=\{i+r.p:$ $ 0\leq r \leq m-1 \}$
of size $m$ for $i=0,\dots,p-1$.
\end{pr}
\begin{proof}
 Given two distinct vertices $a, b \in M_i$ for  $i=0,\dots,p-1$, $a-b\equiv rp[n]$ for some $0<r \leqslant m-1$ ,
then $a-b \notin A$ which proves that each $M_i$ induces a stable sets.

Moreover, it is clear that $\{M_i \}_{i=0,\dots, p-1}$ forms a partition of vertex set of $C(m,p)$.\\
Let us prove that $M_i$ defines a module. For this, suppose that $a=i+r_{a}\cdot p$ and $b=i+r_{b}\cdot p$ two
distinct vertices of a given stable set $M_i$.\\
Let $c \in V\setminus M_i$ such that $ac$  is an edge and let $c=j+r_{c}\cdot p$.\
Let

$ r_{bc}=\left \{ \begin{array}{ll}
     r_b-r_c  & \mbox{if  } r_b> r_c \\
     m+(r_b - r_c) & \mbox{else  }
\end{array}
\right.
$
\hspace{12mm} $r_{ac}= \left \{
\begin{array}{ll}
    r_a-r_c & \mbox{if  }  r_a> r_c \\
    m+(r_a - r_c) &\mbox{else}
\end{array}
\right.
$

two integer numbers such that $b-c\equiv (i-j)+r_{bc}\cdot p[n]$ and
$a-c\equiv (i-j)+r_{ac}\cdot p[n]$ (with $0 \leqslant r_{ac} \leqslant m-1$ and $0 \leqslant r_{bc} \leqslant m-1$.)\\
Since $a-c$ is in $A$ then there is some integers $k$ verifying $0\leqslant k\leqslant r_{ac}$ such that $i-j+kp=p-1$ (or= $p+1$).\\

If $k\leqslant r_{bc}$,  we obtain $b-c\equiv i-j+kp+(r_{bc}-k)\cdot p[n]$.\\
Then $b-c \equiv p-1+(r_{bc}-k)\cdot p[n]$ (or $\equiv p+1+(r_{bc}-k)\cdot p[n]$).
We deduce that $b-c \in A$ since $0\leqslant k \leqslant m-1$.\\

Else, we have $r_{bc} < k \leqslant m+r_{bc}$.
We have $b-c \equiv i-j+r_{bc}\cdot p[n]$.\
Then $b-c \equiv i-j+(m+r_{bc})\cdot p[n]$.
We get $b-c \equiv i-j +kp+(m+r_{bc}-k)\cdot p[n]$ which belongs to $A$ since $0\leqslant m+r_{bc}-k \leqslant m-1$.\\
\end{proof}
\begin{figure}[ht]
\begin{center}
\begin{pspicture}(1,2)(15.5,7.5)
\pscircle(3,5){2}
\pscircle[fillstyle=solid,fillcolor=black](3,7){0.07}
\pscircle[fillstyle=solid,fillcolor=black](3,3){0.07}
\pscircle[fillstyle=solid,fillcolor=red](4.7,6){0.07}
\pscircle[fillstyle=solid,fillcolor=green](4.7,4){0.07}
\pscircle[fillstyle=solid,fillcolor=green](1.3,6){0.07}
\pscircle[fillstyle=solid,fillcolor=red](1.3,4){0.07}
\psline[linestyle=solid](3,7)(4.7,4)
\psline[linestyle=solid](4.7,4)(1.3,4)
\psline[linestyle=solid](1.3,4)(3,7)
\psline[linestyle=solid](4.7,6)(3,3)
\psline[linestyle=solid](3,3)(1.3,6)
\psline[linestyle=solid](1.3,6)(4.7,6)
\uput[45](2,1.8){$C(2,3)$}
\pscircle(11,5){2}
\pscircle[fillstyle=solid,fillcolor=black](11,7){0.07}
\pscircle[fillstyle=solid,fillcolor=black](12.8,4.2){0.07}
\pscircle[fillstyle=solid,fillcolor=black](9.2,4.2){0.07}
\pscircle[fillstyle=solid,fillcolor=blue](11.8,6.8){0.07}
\pscircle[fillstyle=solid,fillcolor=green](12.4,6.4){0.07}
\pscircle[fillstyle=solid,fillcolor=red](12.8,5.8){0.07}
\pscircle[fillstyle=solid,fillcolor=white](13,4.9){0.07}
\pscircle[fillstyle=solid,fillcolor=blue](12.3,3.5){0.07}
\pscircle[fillstyle=solid,fillcolor=green](11.5,3.08){0.07}
\pscircle[fillstyle=solid,fillcolor=red](10.5,3.08){0.07}
\pscircle[fillstyle=solid,fillcolor=white](9.7,3.5){0.07}
\pscircle[fillstyle=solid,fillcolor=blue](9,5){0.07}
\pscircle[fillstyle=solid,fillcolor=green](9.1,5.6){0.07}
\pscircle[fillstyle=solid,fillcolor=red](9.5,6.3){0.07}
\pscircle[fillstyle=solid,fillcolor=white](10.2,6.83){0.07}
\psline[linestyle=solid](11,7)(13,4.9)
\psline[linestyle=solid](11,7)(12.3,3.5)
\psline[linestyle=solid](11,7)(9.7,3.5)
\psline[linestyle=solid](11,7)(9,5)
\psline[linestyle=solid](11.8,6.8)(12.8,4.2)
\psline[linestyle=solid](11.8,6.8)(9.2,4.2)
\psline[linestyle=solid](11.8,6.8)(11.5,3.08)
\psline[linestyle=solid](11.8,6.8)(9.1,5.6)
\psline[linestyle=solid](12.4,6.4)(12.3,3.5)
\psline[linestyle=solid](12.4,6.4)(10.5,3.08)
\psline[linestyle=solid](12.4,6.4)(9,5)
\psline[linestyle=solid](12.4,6.4)(9.5,6.3)
\psline[linestyle=solid](12.8,5.8)(11.5,3.08)
\psline[linestyle=solid](12.8,5.8)(9.7,3.5)
\psline[linestyle=solid](12.8,5.8)(10.2,6.83)
\psline[linestyle=solid](12.8,5.8)(9.1,5.6)
\psline[linestyle=solid](13,4.9)(9.5,6.3)
\psline[linestyle=solid](13,4.9)(10.5,3.08)
\psline[linestyle=solid](13,4.9)(9.2,4.2)
\psline[linestyle=solid](12.8,4.2)(9.7,3.5)
\psline[linestyle=solid](12.8,4.2)(9,5)
\psline[linestyle=solid](12.8,4.2)(10.2,6.83)
\psline[linestyle=solid](12.3,3.5)(9.2,4.2)
\psline[linestyle=solid](12.3,3.5)(9.1,5.6)
\psline[linestyle=solid](11.5,3.08)(9,5)
\psline[linestyle=solid](11.5,3.08)(9.5,6.3)
\psline[linestyle=solid](10.5,3.08)(9.1,5.6)
\psline[linestyle=solid](10.5,3.08)(10.2,6.83)
\psline[linestyle=solid](9.7,3.5)(9.5,6.3)%
\psline[linestyle=solid](9.2,4.2)(10.2,6.83)
\uput[45](10,1.8){$C(3,5)$}
\end{pspicture}
\caption{ Circulant graphs: the vertices of the same color are in the same module.}
\label{weakly}
\end{center}
\end{figure}
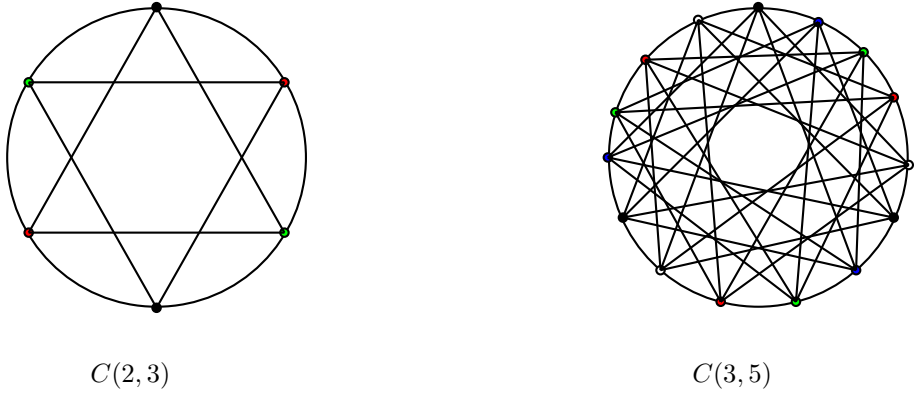

\noindent Since each $M_i$ (for all $0\leqslant i\leqslant p-1$) is a stable set then, by definition of a module, we have:

\begin{pr} \label{permutation}
Any permutation of elements of $M_i$ is an automorphism of $G$ for all $0\leqslant i \leqslant p-1$.
\qed \end{pr}

\noindent By Lemma \ref{module} and Property \ref{proper}, we have $D(C(m,p))\geqslant m$. We will improve this bound:

\begin{theorem} \label{principal}
For all $p \geq 2$ and for all  $m \geq2$, $D(C(m,p)) = m+1$ if $p\neq 4$.
\end{theorem}
\section{Proof of Theorem \ref{main} and Theorem  \ref{principal}}\label{sec:3}
In this section, we give the proof of Theorem \ref{principal} in the first step,
while the second step is spent to give the proof of the Theorem \ref{main}

\begin{lemma} \label{borne}
For all $p \geq 2$ and for all  $m \geq2$, $D(C(m,p)) > m$.
\end{lemma}
\begin{proof}
 \noindent If $p=2$ (resp. $p=3)$ then $C(m,2)\cong K_{m,m}$ (resp. $C(m,3)\cong K_{m,m,m}$).
According to Theorem \ref{multipartite}, we have $D(C(m,p))>m$.
Let $C(m,p)$ be the circulant graph generated by $A=\{p-1+rp, p+1+rp: 0\leqslant r \leqslant m-1\}$.

\noindent Let us suppose that $p>3$. Since the modules $M_i$ $(i=0,\dots, p-1)$ are stables of size $m$, then
by Lemma \ref{module} we have $D(C(m,p))\geq m$.\\
Consider $c:V(C(m,p))\rightarrow \{1,2,\dots,m\}$ be a $m$-labeling of $C(m,p)$ $(m \geq 2)$
and prove that $c$ is not $m$-distinguishing.\\
By way of contradiction, assume that $c$ is $m$-distinguishing.

\noindent For all distinct vertices $v$, $w$ in a given module $M_{i_0}$ with $i_0\in \{0,1,\dots,p-1\}$ we have
$c(v)\neq c(w)$ otherwise, there exists a transposition $\tau$ of $v$ and $w$ verifying $c=c \circ \tau$.
This yields a contradiction. That means that
in a fixed module $M_i$ we have all labels.

\noindent Let $P_j$  ($1\leqslant j \leqslant m$) be a set of index $\{(j-1)p+i,i \in \{0, \dots, p-1\} \}$.

\noindent Let $v\in M_i$ ( $0\leqslant i\leqslant p-1$) then $v=i+rp$ where $0\leqslant r \leqslant m-1$.
\noindent  Consider now the mapping $\delta_i$ with $i=0,\dots,p-1$ defined as follows:
$\delta_i: V \rightarrow V$ such that $\delta_i(v)=(c(v)-1)p+i$ if $v \in M_i$ else $\delta (v)=v$. By Property \ref{permutation}, $\delta_i$
defines an automorphism of $G$.\\
\noindent Let  $\delta = \delta_0 \circ \dots \circ \delta_{p-1}$ be an automorphism of $G$.

\noindent Let $\psi$ be a mapping defined as follows:
$\psi: V \rightarrow V$ such that $\psi(i+rp)= p-(i+1) + rp$. Let prove that $\psi$ is an automorphism of $G$.

\noindent Let $a=i+rp$ and $b=j+r'p$ two adjacent vertices then $b-a=j-i+(r'-r)p \in A$. We have $\psi(b)- \psi(a) = i-j +(r'-r)p$ which belongs to $A$. Thus $\psi$ is an automorphism of $G$.

\noindent Check now that $\delta ^{-1} \circ \psi \circ \delta$ is non trivial automorphism of $G$ preserving the labeling  $c$. See Fig. \ref{composition}.

\noindent Then $\delta ^{-1} \circ \psi \circ \delta$ is clearly an automorphism because it is a composition of automorphisms.

\noindent Since $\delta ^{-1} \circ \psi \circ \delta(0) = \delta ^{-1} \circ \psi ( (c(0)-1)p +0)
   = \delta ^{-1} ( (c(0)-1)p+(p-1)) = u$  with $u \in M_{p-1}$ and $c(u)=c(0)$, then $u\neq 0$ since $0 \in M_0$
and $M_0 \neq M_{p-1}$ and  $p> 1$.
\noindent Thus $\delta ^{-1} \circ \psi \circ \delta$ is not a trivial automorphism.

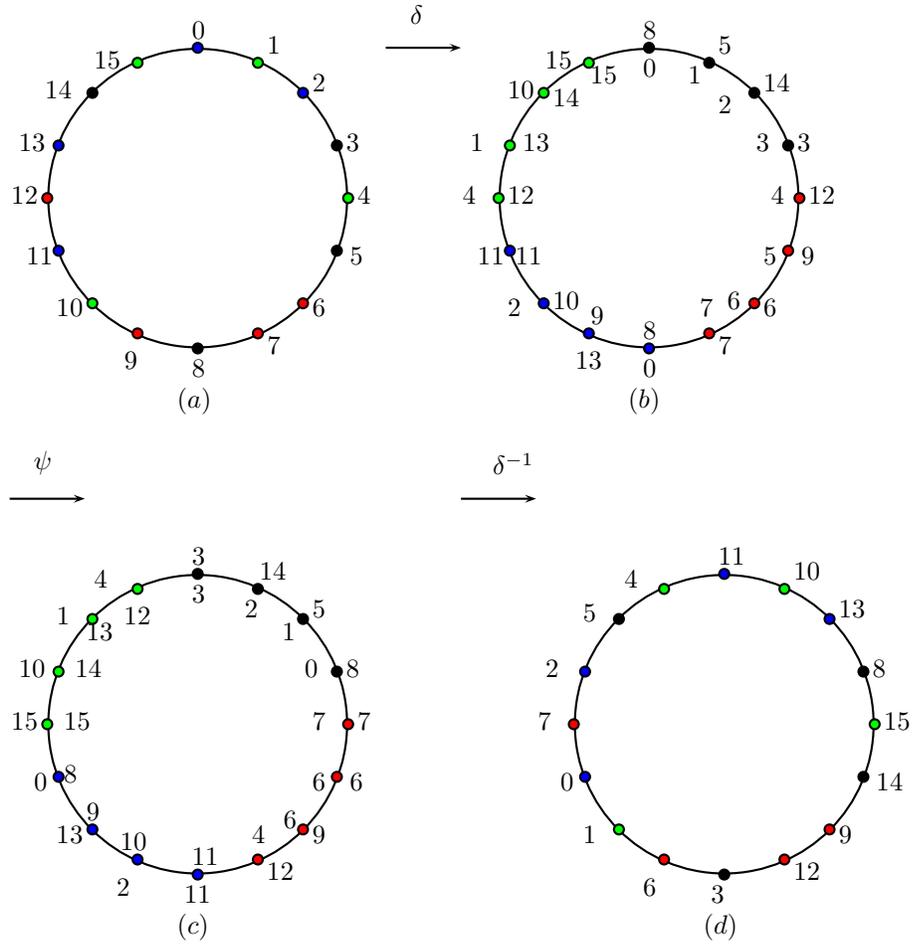
\begin{figure}
\begin{center}
\begin{pspicture}(0,2)(18,14)
\pscircle(3,11){2}
\pscircle[fillstyle=solid,fillcolor=red](1,11){0.08}\uput[45](0.4,10.8){$12$}
\pscircle[fillstyle=solid,fillcolor=green](5,11){0.08}\uput[45](5,10.8){$4$}
\pscircle[fillstyle=solid,fillcolor=black](3,9){0.08}\uput[45](2.8,8.5){$8$}
\pscircle[fillstyle=solid,fillcolor=blue](3,13){0.08} \uput[45](2.8,13){$0$}
\pscircle[fillstyle=solid,fillcolor=green](1.6,9.6){0.08}\uput[45](1,9.3){$10$}
\pscircle[fillstyle=solid,fillcolor=red](4.4,9.6){0.08}\uput[45](4.4,9.3){$6$}
\pscircle[fillstyle=solid,fillcolor=black](1.6,12.4){0.08}\uput[45](0.85,12.2){$14$}
\pscircle[fillstyle=solid,fillcolor=blue](4.4,12.4){0.08}\uput[45](4.4,12.3){$2$}
\pscircle[fillstyle=solid,fillcolor=red](2.2,9.2){0.08}\uput[45](1.9,8.6){$9$}
\pscircle[fillstyle=solid,fillcolor=red](3.8,9.2){0.08}\uput[45](3.8,8.8){$7$} 
\pscircle[fillstyle=solid,fillcolor=green](2.2,12.8){0.08}\uput[45](1.5,12.6){$15$} 
\pscircle[fillstyle=solid,fillcolor=green](3.8,12.8){0.08}\uput[45](3.8,12.8){$1$}
\pscircle[fillstyle=solid,fillcolor=black](4.85,11.7){0.08}\uput[45](4.85,11.5){$3$}
\pscircle[fillstyle=solid,fillcolor=black](4.85,10.3){0.08}\uput[45](4.9,10){$5$}
\pscircle[fillstyle=solid,fillcolor=blue](1.15,10.3){0.08}\uput[45](0.6,10){$11$}
\pscircle[fillstyle=solid,fillcolor=blue](1.15,11.7){0.08}\uput[45](0.5,11.5){$13$}
\uput[45](2.6,8){$(a)$}
\psline{->}(5.5,13)(6.5,13)
\uput[45](5.7,13.2){$\delta$}
\pscircle(9,11){2}
\pscircle[fillstyle=solid,fillcolor=green](7,11){0.08} \uput[45](6.4,10.8){$4$} \uput[45](7,10.8){$12$} 
\pscircle[fillstyle=solid,fillcolor=red](11,11){0.08} \uput[45](11,10.8){$12$} \uput[45](10.5,10.8){$4$}
\pscircle[fillstyle=solid,fillcolor=blue](9,9){0.08} \uput[45](8.8,8.5){$0$} \uput[45](8.8,9){$8$}
\pscircle[fillstyle=solid,fillcolor=black](9,13){0.08} \uput[45](8.8,13){$8$} \uput[45](8.8,12.5){$0$}
\pscircle[fillstyle=solid,fillcolor=blue](7.6,9.6){0.08} \uput[45](7,9.3){$2$} \uput[45](7.6,9.4){$10$}
\pscircle[fillstyle=solid,fillcolor=red](10.4,9.6){0.08} \uput[45](10.4,9.3){$6$} \uput[45](9.92,9.4){$6$}
\pscircle[fillstyle=solid,fillcolor=green](7.6,12.4){0.08} \uput[45](7,12.2){$10$} \uput[45](7.6,12.1){$14$}
\pscircle[fillstyle=solid,fillcolor=black](10.4,12.4){0.08} \uput[45](10.4,12.3){$14$} \uput[45](9.8,12){$2$}
\pscircle[fillstyle=solid,fillcolor=blue](8.2,9.2){0.08} \uput[45](7.9,8.6){$13$} \uput[45](8.1,9.2){$9$}
\pscircle[fillstyle=solid,fillcolor=red](9.8,9.2){0.08} \uput[45](9.8,8.8){$7$} \uput[45](9.55,9.3){$7$}
\pscircle[fillstyle=solid,fillcolor=green](8.2,12.8){0.08} \uput[45](7.5,12.6){$15$} \uput[45](8.1,12.4){$15$}
\pscircle[fillstyle=solid,fillcolor=black](9.8,12.8){0.08} \uput[45](9.8,12.8){$5$} \uput[45](9.4,12.4){$1$}
\pscircle[fillstyle=solid,fillcolor=black](10.85,11.7){0.08} \uput[45](10.85,11.5){$3$} \uput[45](10.3,11.5){$3$} 
\pscircle[fillstyle=solid,fillcolor=red](10.85,10.3){0.08} \uput[45](10.9,10){$9$} \uput[45](10.4,10){$5$}
\pscircle[fillstyle=solid,fillcolor=blue](7.15,10.3){0.08} \uput[45](6.6,10){$11$} \uput[45](7.1,10){$11$}
\pscircle[fillstyle=solid,fillcolor=green](7.15,11.7){0.08} \uput[45](6.5,11.5){$1$} \uput[45](7.2,11.5){$13$}
\uput[45](8.6,8){$(b)$}
\psline{->}(0.5,7)(1.5,7)
\uput[45](0.7,7.2){$\psi$}
\pscircle(3,4){2}
\pscircle[fillstyle=solid,fillcolor=green](1,4){0.08} \uput[45](0.4,3.8){$15$} \uput[45](1.1,3.8){$15$}
\pscircle[fillstyle=solid,fillcolor=green](1.15,4.7){0.08} \uput[45](0.5,4.5){$10$} \uput[45](1.25,4.5){$14$}
\pscircle[fillstyle=solid,fillcolor=green](1.6,5.4){0.08} \uput[45](1,5.2){$1$} \uput[45](1.4,5){$13$}
\pscircle[fillstyle=solid,fillcolor=green](2.2,5.8){0.08} \uput[45](1.5,5.7){$4$} \uput[45](1.9,5.2){$12$}
\pscircle[fillstyle=solid,fillcolor=red](5,4){0.08} \uput[45](5,3.8){$7$} \uput[45](4.4,3.8){$7$}
\pscircle[fillstyle=solid,fillcolor=red](4.85,3.3){0.08} \uput[45](4.9,3){$6$} \uput[45](4.4,3){$6$}
\pscircle[fillstyle=solid,fillcolor=red](4.4,2.6){0.08} \uput[45](4.4,2.3){$9$} \uput[45](4.02,2.5){$6$}
\pscircle[fillstyle=solid,fillcolor=red](3.8,2.2){0.08} \uput[45](3.8,1.8){$12$} \uput[45](3.6,2.3){$4$}
\pscircle[fillstyle=solid,fillcolor=blue](3,2){0.08} \uput[45](2.7,1.5){$11$} \uput[45](2.8,2){$11$}
\pscircle[fillstyle=solid,fillcolor=blue](2.2,2.2){0.08} \uput[45](1.8,1.6){$2$}\uput[45](1.85,2.2){$10$}
\pscircle[fillstyle=solid,fillcolor=blue](1.6,2.6){0.08} \uput[45](1,2.3){$13$}\uput[45](1.4,2.6){$9$}
\pscircle[fillstyle=solid,fillcolor=blue](1.15,3.3){0.08} \uput[45](0.7,3){$0$} \uput[45](1.1,3.1){$8$}
\pscircle[fillstyle=solid,fillcolor=black](3.8,5.8){0.08} \uput[45](3.7,5.8){$14$} \uput[45](3.5,5.3){$2$}
\pscircle[fillstyle=solid,fillcolor=black](3,6){0.08} \uput[45](2.8,6){$3$} \uput[45](2.8,5.5){$3$}
\pscircle[fillstyle=solid,fillcolor=black](4.4,5.4){0.08} \uput[45](4.4,5.3){$5$} \uput[45](4,5){$1$}
\pscircle[fillstyle=solid,fillcolor=black](4.85,4.7){0.08} \uput[45](4.85,4.5){$8$} \uput[45](4.3,4.5){$0$}
\uput[45](6.8,7.2){$\delta^{-1}$}
\psline{->}(6.5,7)(7.5,7)
\uput[45](2.6,1){$(c)$}
\pscircle(10,4){2}
\pscircle[fillstyle=solid,fillcolor=red](8,4){0.08} \uput[45](7.4,3.8){$7$} 
\pscircle[fillstyle=solid,fillcolor=green](12,4){0.08} \uput[45](7.5,4.5){$2$} 
\pscircle[fillstyle=solid,fillcolor=black](10,2){0.08} \uput[45](8,5.25){$5$} 
\pscircle[fillstyle=solid,fillcolor=blue](10,6){0.08} \uput[45](8.55,5.7){$4$} 
\pscircle[fillstyle=solid,fillcolor=green](8.6,2.6){0.08} \uput[45](12,3.8){$15$} 
\pscircle[fillstyle=solid,fillcolor=red](11.4,2.6){0.08} \uput[45](11.9,3){$14$} 
\pscircle[fillstyle=solid,fillcolor=black](8.6,5.4){0.08} \uput[45](11.4,2.3){$9$}
\pscircle[fillstyle=solid,fillcolor=blue](11.4,5.4){0.08} \uput[45](10.8,1.8){$12$} 
\pscircle[fillstyle=solid,fillcolor=red](9.2,2.2){0.08} \uput[45](9.7,1.5){$3$} 
\pscircle[fillstyle=solid,fillcolor=red](10.8,2.2){0.08} \uput[45](8.8,1.6){$6$} 
\pscircle[fillstyle=solid,fillcolor=green](9.2,5.8){0.08} \uput[45](8,2.3){$1$} 
\pscircle[fillstyle=solid,fillcolor=green](10.8,5.8){0.08} \uput[45](7.7,3){$0$} 
\pscircle[fillstyle=solid,fillcolor=black](11.85,4.7){0.08} \uput[45](10.8,5.8){$10$} 
\pscircle[fillstyle=solid,fillcolor=black](11.85,3.3){0.08} \uput[45](9.8,6){$11$} 
\pscircle[fillstyle=solid,fillcolor=blue](8.15,3.3){0.08} \uput[45](11.4,5.3){$13$} 
\pscircle[fillstyle=solid,fillcolor=blue](8.15,4.7){0.08} \uput[45](11.85,4.5){$8$} 
\uput[45](9.6,1){$(d)$}
\end{pspicture}
\vspace{5mm}\caption{The automorphism $\delta^{-1}\circ\psi\circ\delta$ applied to $C(4,4)$ with four labels (1,2,3,4)=(black,red, blue,green).}
\label{composition}
\end{center}
\end{figure}

\noindent To complete the proof, it is enough to show that $c(u)=c(\delta ^{-1} \circ \psi \circ \delta(u))$ for all vertex $u$.\\
\noindent Let $u=i+rp$ then we have $\delta ^{-1} \circ \psi \circ \delta (u) =\delta ^{-1} \circ \psi ((c(u)-1)p +i)=
\delta ^{-1} ((c(u)-1)p+ p-(i+1)) =v$ such that $v\in M_{p-(i+1)}$ and $c(v)=c(u)$.

\noindent Then $\delta ^{-1} \circ \psi \circ \delta$ preserves the labeling.

\end{proof}

\vspace{3mm}\noindent The following result gives the exact value of $D(C(m,p))$

\begin{lemma}\label{D(G)}
For all $p\geq 2$ and $p\neq4$ and for all $m \geq 2$ : $D(C(m,p)) \leqslant m+1$
\end{lemma}
\begin{proof}
 If $p\in \{2,3\}$ the proposition is true by Theorem \ref{multipartite}. Consider $c$ be the $(m+1)$-labeling defined as follows (See Fig. \ref{m+1color}):


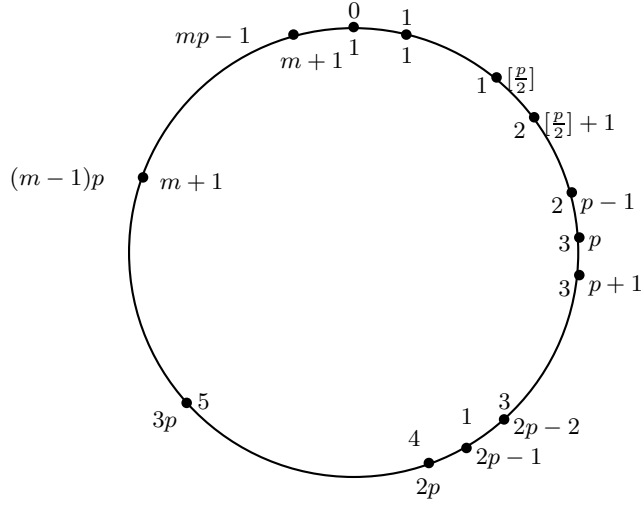
\begin{figure}
\centering
\begin{pspicture}(0,1)(9,9)
\pscircle(5,5){3}
\pscircle[fillstyle=solid,fillcolor=black](5,8){0.07}
\pscircle[fillstyle=solid,fillcolor=black](5.7,7.9){0.07}
\pscircle[fillstyle=solid,fillcolor=black](6.9,7.33){0.07}
\pscircle[fillstyle=solid,fillcolor=black](7.4,6.8){0.07}
\pscircle[fillstyle=solid,fillcolor=black](7.9,5.8){0.07}
\pscircle[fillstyle=solid,fillcolor=black](8,5.2){0.07}
\pscircle[fillstyle=solid,fillcolor=black](8,4.7){0.07}
\pscircle[fillstyle=solid,fillcolor=black](7,2.78){0.07}
\pscircle[fillstyle=solid,fillcolor=black](6.5,2.4){0.07}
\pscircle[fillstyle=solid,fillcolor=black](6,2.2){0.07}
\pscircle[fillstyle=solid,fillcolor=black](2.78,3){0.07}
\pscircle[fillstyle=solid,fillcolor=black](2.2,6){0.07}
\pscircle[fillstyle=solid,fillcolor=black](4.2,7.9){0.07}
\uput[45](4.8,8){\small $0$}
\uput[45](5.5,7.9){\small $1$}
\uput[45](6.9,7){\small $[\frac{p}{2}]$}
\uput[45](7.4,6.4){\small $[\frac{p}{2}]+1$}
\uput[45](7.9,5.4){\small $p-1$}
\uput[45](8,4.9){\small $p$}
\uput[45](8,4.3){\small $p+1$ }
\uput[45](7,2.4){\small $2p-2$}
\uput[45](6.5,2){\small $2p-1$}
\uput[45](5.7,1.6){\small $2p$}
\uput[45](2.2,2.5){\small $3p$}
\uput[45](0.3,5.7){\small $(m-1)p$}
\uput[45](2.5,7.6){\small $mp-1$}
\uput[45](4.8,7.5){\small $1$}
\uput[45](5.5,7.4){\small $1$}
\uput[45](6.5,7){\small $1$}
\uput[45](7,6.4){\small $2$}
\uput[45](7.5,5.4){\small $2$}
\uput[45](7.6,4.9){\small $3$}
\uput[45](7.6,4.3){\small $3$ }
\uput[45](6.8,2.8){\small $3$}
\uput[45](6.3,2.6){\small $1$}
\uput[45](5.6,2.3){\small $4$}
\uput[45](2.8,2.8){\small $5$}
\uput[45](2.3,5.7){\small $m+1$}
\uput[45](3.9,7.3){\small $m+1$}
\end{pspicture}
\caption{ The $(m+1)$-labeling: the label of each vertex is given inside the cycle.} \label{m+1color}
\end{figure}


\begin{equation*}
c(v)=    \left\{
      \begin{array}{ll}
        1   & \hspace{7mm}   0 \leqslant v \leqslant \lfloor \frac{p}{2}\rfloor  \hspace{2mm} \text{and}\hspace{2mm} v=2p-1  \\
        2   & \hspace{7mm}   \lfloor \frac{p}{2}\rfloor < v \leqslant p-1 \\
        j+1 & \hspace{7mm}   v\in P_j  \hspace{2mm} \text{and} \hspace{2mm}2\leqslant j\leqslant m \hspace{2mm}\text{and} \hspace{2mm} v\neq 2p-1

      \end{array}
    \right.
\end{equation*}

\noindent Suppose that there exists an automorphism $\delta$ preserving this labeling and prove that $\delta$ is trivial.

\noindent Since $p>4$, $0$ is the unique vertex labeled $1$ which has the following sequence of label
in his neighborhood $(1,1,2,3,4,4,\dots,m+1,m+1)$. Thus $\delta(0)=0$.

\noindent However, we refer to the following claim:
\begin{claim}\label{distance}
For each vertex  $i$ in $C(m,p)$ where $0\leq i\leq p-1$, we have:

\begin{equation*}
d(0,i)=    \left\{
       \begin{array}{ll}
         i   &   \hspace{5mm} 1 \leqslant i \leqslant \lfloor \frac{p}{2}\rfloor \\
         p-i &   \hspace{5mm} \lfloor \frac{p}{2} \rfloor <  i \leqslant p-1
       \end{array}
     \right.
\end{equation*}
 \end{claim}
\begin{proof}
\noindent First observe that  for all pair of vertices $u$ and $v$ in the same module $M$ and $z\in V\setminus M$,
we have $d(u,z)=d(v,z)$ and $d(u,v)=2$.

Now, if we contract each module $M_i$ of $C(m,p)$, then we get a cycle on $p$ vertices which implies the claim.
\end{proof}

\noindent Let us prove that each vertex lebeled $1$, is fixed by the automorphism $\delta$:\\
Consider the table describing the sequence of labels of the vertex $u$:
\begin{table}
\begin{tabular}{lll}
\hline\noalign{\smallskip}
\bf $u$ & \bf $c(u)$ & \bf $c(N(u))$ \\
\noalign{\smallskip}\hline\noalign{\smallskip}
\bf $0$ & \bf $1$ & \bf $1,1,2,3,4,4, \dots, m+1,m+1$.\\
\bf $0 < i < \lfloor \frac{p}{2}\rfloor$ & \bf $1$ & \bf $1,1,3,3,4,4, \dots, m+1,m+1$.\\
\bf $\lfloor \frac{p}{2}\rfloor$ & \bf $1$ & \bf $1,2,3,3,4,4, \dots, m+1,m+1$.\\
\bf $\lfloor \frac{p}{2}\rfloor < j < p-1$ & \bf  $2$ & \bf $2, 2 ,3,3,4,4, \dots, m+1,m+1$.\\
\bf $p-1$ & \bf $2$ & \bf $1, 2, 3, 3,4,4, \dots, m+1,m+1$.\\
\bf $2p-1$ & \bf $1$ & \bf $1,2,3,3,4,4, \dots, m+1,m+1$.\\
\noalign{\smallskip}\hline
\end{tabular}
\caption{The sequence of labels being in the neighborhood of vertices.}
\end{table}

For all $i$ such that $0< i< \lfloor \frac{p}{2} \rfloor$, we have the sequence of labels occurring in the neighborhood
of a vertex $i$ is $(1,1,3,3, \dots m+1, m+1)$. More than, for all two distinct vertices $u$ and $v$ such
that $0< u,v< \lfloor \frac{p}{2} \rfloor$ we have $d(u,0)\neq d(v,0)$.
Then, since $\delta(0)=0$ we get $\delta (u)=u$ and $\delta (v)=v$.
Generally, for all vertex $i$ such that $0< i< \lfloor \frac{p}{2} \rfloor$, we obtain $\delta(i)=i$.\\

\noindent More than, the sequence of labels in the neighborhood of $2p-1$ and $\lfloor \frac{p}{2}\rfloor$ is
$\{1, 2, 3, 3, 4, 4, \dots, m+1, m+1 \}$. Since $d(\lfloor \frac{p}{2} \rfloor,0) > d(2p-1,0)=1$, then we get $\delta(2p-1)=2p-1$ and
$\delta(\lfloor \frac{p}{2} \rfloor)= \lfloor \frac{p}{2} \rfloor$.

\noindent Now observe that by the previous claim, any distinct vertices $u$ and $v$ labeled $2$, we have $d(u,0)\neq d(v,0)$.
Then for any vertex $u$ such that $c(u)=2$, we have $\delta(u)=u$.

Finally, let us prove that each vertex $v$ in $C(m,p)\setminus (P_1\cup \{2p-1\})$ is fixed by the automorphism $\delta$.
For that, it is enough to show  for all pair of distinct vertices $u$ and $v$ such that $c(u)=c(v)$,  we have
$N(u)\cap \{0,1,2,\dots,p-1\} \neq  N(v)\cap \{0,1,2,\dots,p-1\}$. This proposition will imply that each vertex $v$
labeled $c(v)$  $(c(v)\geq2)$ is fixed by $\delta$ and we conclude the proof of theorem.

\noindent Let $u$ and $v$ two distinct vertices such that $c(u)=c(v)$ with $u,v \in C(m,p)\setminus (P_1\cup \{2p-1\}) $.\\

\noindent Since $c(u)=c(v)$, we have $u \in M_i$ and $v\in M_j$ with $i\neq j$.
Then $i-1, i+1 \in N(u)$ and $j-1, j+1 \in N(v)$.

If $i=0$ then $p-1\in N(u)$ since $p\in M_i$.
Similarly, if $i=p-1$, then $0\in N(u)$ since $mp-1\in M_i$.

\noindent Therefore, modulo $p$, we have that $i-1, i+1 \in N(u)\cap \{0,1,\dots,p-1 \}$ and $j-1, j+1 \in N(v)\cap \{0,1,\dots,p-1 \}$.

Additionally, observe that any vertex $u$ has exactly two neighborhood among $p$ consecutive vertices of $G$.
Thus $N(u)\cap \{0,1, \dots,p-1\} =\{i-1, i+1 \; \; \bmod{p} \}$ and $N(v)\cap \{0,1, \dots,p-1\} =\{j-1, j+1\; \; \bmod{p}\}$.

\noindent Now, if $N(u)\cap \{0,1,\dots,p-1 \}= N(v)\cap \{0,1,\dots,p-1 \}$ and $i\neq j$, then $i+1=j-1$ and $i-1=j+1$.
Thus $j=i-2$, $j=i+2$ and $p=4$.

Since $p>4$, we get that $N(u)\cap \{0,1,\dots,p-1 \}\neq N(v)\cap \{0,1,\dots,p-1 \}$.
\end{proof}
\vspace{3mm}

\noindent Lemma \ref{borne} and Lemma \ref{D(G)} give the proof of Theorem \ref{principal}.
The following result gives the value of distinguishing number for $p=4$:

\begin{corollary}\label{p4}
For each $m\geq 2$, $C(m,4)$ is
isomorphic to $C(2m,2)$ $($or $K_{2m,2m})$ and $D(C(m,4))=$ $2m+1$.
\end{corollary}
\begin{proof}
 \noindent The graph $C(m,4)$ is partitioned into four modules $M_0$, $M_1$, $M_2$, $M_3$. We have:
$N(M_0)=N(M_2)=M_1\cup M_3$  and $N(M_1)=N(M_3)=M_0\cup M_2$. Thus, the module $M_i$
is not maximal where $i \in \{0,1,2,3\}$. Furthermore, $M_0 \cup M_2$  and $M_1\cup M_3$ are stables of size $2m$.
Then, the graph $C(m,4)$ is a multipartite graph $K_{2m,2m}$ and $D(C(m,4))=D(K_{2m,2m})=D(C(2m,2))=2m+1$.
\end{proof}

\noindent \textbf{PROOF OF THEOREM \ref{main}}

\noindent Let $d_1,d_2,\dots,d_r$ be an ordered sequence of distinct integers. Let $m_i=d_i -1$ for all $i=1,\dots,r$
and $p_i=\displaystyle\prod_{j\neq i} m_j$.

\noindent By definition, $m_i p_i=m_j p_j$ for $i\neq j$ for $i,j=1,\dots,r$.\\
If all $p_i\neq 4$, then let $n=m_i p_i$ else $n=3m_i p_i$ for all $i=1,\dots,r$.\\
Now, by Theorem \ref{principal}, $D(C(m_i,p_i))=m_i+1=d_i$ for all $i=1,\dots,r$.\\
So, $(G_i)_i ={(C(m_{i},p_{i}))}_i$ with $i=1,\dots,r$,
is a family of connected circulant graphs of order $n$ such that $D(G_i)=d_i$.
\qed

\section{Remarks and conclusion}\label{sec:4}

\noindent We have studied the structure of circulant graphs $C(m,p)$ by providing the associated distinguishing number.
We have  determined the distinguishing number of circulant graphs $C(m,p)$ for all $m.p\geq 3$ with $m\geqslant 1$ and $p\geqslant 2$.
We can summarize the result which give the value of distinguishing number for circulant graphs $C(m,p)$ as follows:

$D(C(m,p))=$
$\begin{cases}
m &  (m\geqslant 3 \; \; \text{and} \; \; p=1)\\
m+1 & (m=1 \; \; \text{and} \; \; p\geq 6) \; \; or \; \; (m\geq2 \; \; p\geq2 \; \; p\neq4) \\
2m+1 & (m=1 \; \; \text{and} \; \; p\in\{3,4,5\}) \; \; or \; \; (m\geq2  \; \; p=4)
\end{cases}$

\noindent We deduce that for a given integer $n=\displaystyle\prod_{i=1}^{r} m_i$ for $r\geq 2$ and $m_i\geq 1$,
we can build a family of graphs of same order $n$ where the distinguishing number
depends on divisors of $n$ . The main idea of constructing such graphs consists of partitioning the vertex set into
modules of same size. The circulant graphs are well privileging structure. One may ask if we can construct such family
of circulant graphs with smaller order?

\noindent For instance, we can improve in  Theorem \ref{main} the order  $n$ of $(C(m_i,p_i))_i$ for $i=1,\dots r$,
by taking $n=\frac{\displaystyle\prod_{i=1}^{r} m_i}{gcd(m_i, \displaystyle\prod_{j<i} m_j)}$.

\begin{thebibliography}{}
\bibitem{albertson} M.~O. Albertson and K.~L. Collins.  Symmetry breaking in graphs.
 {\em Electronic J. of Combinatorics}. {\bf 3}(1996),\# R18.

\bibitem{Bogstad} B.~Bogstad  and L.~Cowen.   The distinguishing number of hypercubes.
{\em Discrete Mathematics}. {\bf383}(2004),29--35, .

\bibitem{tree} C.~T. Cheng.  \newblock
  On computing the distinguishing numbers of trees and forests.  {\em Electronic
J. of Combinatorics}. {\bf13}(2011),\# R11.

\bibitem{chrom} K.~L. Collins and A.~N. Trenk.   The Distinguishing Chromatic Number.
 {\em Electronic J. of Combinatorics.} {\bf 13}(2006),\# R16.

\bibitem{Fisher_1} M.~J. Fisher and G.~Isaak.
  Distinguishing colorings of Cartesian products of
complete graphs.  {\em Discrete Mathematics}. {\bf308}(2008),2240--2246.

\bibitem{Fisher_2} M.~J. Fisher and G.~Isaak.   Distinguishing numbers of Cartesian products of
multiple complete graphs.   {\em PARS Mathematica Comptemporanea.} {\bf5}(2012),159--170.

\bibitem{gravier} S.~Gravier, J.~Jerebic and M.~Mollard.
   Distinguishing number of some circulant graphs.  {\em Manuscript}. 2010.

\bibitem{klav_cliques} W.~Imrich, J.~Jerebic and  S.~Klav\v{z}ar.
  The distinguishing number of Cartesian products
of complete graphs.   {\em European. J. Combin.} {\bf45}(2009), 175--188.

\bibitem{Imrich_cartes_power} W.~Imrich and S.~Klav\v{z}ar.
Distinguishing Cartesian powers of graphs. {\em J. Graph
    Theory}. {\bf53}(2006),250--260.

\bibitem{klav_power} S.~Klav\v{z}ar and X.~Zhu.   Cartesian powers of graphs can be distinguished by
two labels.   {\em European J. Combinatorics}. {\bf 28} (2007) 303--310.

\bibitem{Potanka} K.~S. Potanka.  Groups, Graphs and Symmetry Breaking.  Masters Thesis,
 Virginia Polytechnic Institute and State University, 1998.

\bibitem{rudin} F.~Rubin.  Problem 729 in {\em J. Recreational Math. {\bf Vol. 11} (Solution in
Vol.12, 1980)}(1979),128.

\bibitem{tym} J.~Tymoczko.  Distinguishing number for graphs and groups.
{\em Electronic J. Combinatorics}, {\bf11(1)}(2004),\# R63.(Also available at arXiv:math.CO/0406542.).

\bibitem{Z} X.~Zhu and T.~L. Wong. Distinguishing labeling of group actions.
 {\em Discrete Mathematics, Vol 309.} {\bf6} (2009),1760--1765.

\end{thebibliography}

\end{document}